\documentclass[english,reqno]{amsart}
\usepackage[T1]{fontenc}
\usepackage[latin1]{inputenc}
\usepackage{babel}

\usepackage{prettyref}
\usepackage{textcomp}
\usepackage{amsthm}
\usepackage{graphicx}
\usepackage{amssymb}
\usepackage{esint}
\usepackage[unicode=true, pdfusetitle,
 bookmarks=true,bookmarksnumbered=false,bookmarksopen=false,
 breaklinks=false,pdfborder={0 0 1},backref=false,colorlinks=false]
 {hyperref}

\numberwithin{equation}{section} 
\numberwithin{figure}{section} 
\theoremstyle{plain}
\newtheorem{thm}{Theorem}[section]
  \theoremstyle{definition}
  \newtheorem{defn}[thm]{Definition}
  \theoremstyle{plain}
  \newtheorem{lem}[thm]{Lemma}
  \theoremstyle{remark}
  \newtheorem{rem}[thm]{Remark}
 \theoremstyle{definition}
  \newtheorem{example}[thm]{Example}
  \theoremstyle{plain}
  \newtheorem{prop}[thm]{Proposition}
  \theoremstyle{remark}
  \newtheorem*{acknowledgement*}{Acknowledgement}

\usepackage{lmodern}

\begin{document}

\title{On Explicit Holmes-Thompson Area Formula in Integral Geometry}

\author{Yang Liu}

\date{October 29, 2009}
\begin{abstract}
In this article, we give an exposition on the Holmes-Thompson theory
developed by Alvarez. The space of geodesics in Minkowski space has
a symplectic structure which is induced by the projection from the
sphere-bundle. we show that it can be also obtained from the symplectic
structure on the tangent bundle of the Riemannian manifold, the tangent
bundle of the Minkowski unit sphere. We give detailed descriptions
and expositions on Holmes-Thompson volumes in Minkowski space by the
symplectic structure and the Crofton measures for them. For the Minkowski
plane, a normed two dimensional space, we express the area explicitly
in an integral geometry way, by putting a measure on the plane, which
gives an extension of Alvarez's result for higher dimensional cases.
\end{abstract}

\keywords{Minkowski space, Holmes-Thompson Volume, symplectic structure, convex
valuation}

\subjclass[2000]{28A75, 51A50, 32F17, 53B40}

\maketitle

\section{Introductions\label{sec:Introduction-on-Minkowski}}

\subsection{Minkowski Space and Geodesics}

A Minkowski space is a vector space with a Minkowski norm, and a Minkowski
norm is defined in \cite{CS} as
\begin{defn}
\label{def:A-smooth-function}A function $F:\mathbb{R}^{n}\rightarrow\mathcal{\mathbf{\mathbb{R}}}$
is a Minkowski norm if\end{defn}
\begin{enumerate}
\item $F(x)>0$ for any $x\in\mathbb{R}^{n}\setminus\left\{ 0\right\} $
and $F(0)=0$. 
\item $F(\lambda x)=|\lambda|F(x)$ for any $x\in\mathbb{R}^{n}\setminus\left\{ 0\right\} $.
\label{enu:posi}
\item $F\in C^{\infty}(\mathbb{R}^{n}\setminus\left\{ 0\right\} )$ and
the symmetric bilinear form \begin{equation}
g_{x}(u,v):=\frac{1}{2}\frac{\partial^{2}}{\partial s\partial t}F^{2}(x+su+tv)|_{s=t=0}\end{equation}
 is positively definite on $\mathbb{R}^{n}$ for any $x\in\mathbb{R}^{n}\setminus\left\{ 0\right\} $.\label{enu:hes}
\end{enumerate}
We denote a Minkowski space by $(\mathbb{R}^{n},F)$. By the way,
(\prettyref{enu:posi}) and (\prettyref{enu:hes}) in Definition \prettyref{def:A-smooth-function}
imply the the convexity of $F$, see Chapter 1 of \cite{CS}.

First of all, we can infer the following theorem about geodesics in
Minkowski space from Definition \prettyref{def:A-smooth-function}.
\begin{thm}
\label{thm:The-straight-lines}The straight line joining two points
in Minkowski space is the only shortest curve joining them.\end{thm}
\begin{proof}
For any $p,q\in(\mathbb{R}^{n},F)$, let $r(t),\mbox{\,}t\in[a,b]$
with $F(r'(t))=1$, be a curve joining $p$ and $q$, which has the
minimum length. Then $r(t)$ is the minimizer of the functional $\int_{a}^{b}F(r'(t))dt$. 

Note that $F$ is smooth. By the fundamental lemma of calculus of
variation (Let $V(h):=\int_{a}^{b}F(r'(t)+h\delta'(t))dt$, where
$\delta(a)=\delta(b)=0$. Then\begin{equation}
\begin{array}{lll}
V'(0) & = & \frac{\partial}{\partial h}|_{h=0}\int_{a}^{b}F(r'(t)+h\delta'(t))dt\\
 & = & \int_{a}^{b}\frac{\partial}{\partial h}|_{h=0}F(r'(t)+h\delta'(t))dt\\
 & = & \int_{a}^{b}\nabla F(r'(t))\cdot\delta'(t)dt\\
 & = & \nabla F(r'(t))\cdot\delta(t)|_{a}^{b}-\int_{a}^{b}\delta(t)\cdot\frac{\mathrm{d}}{\mathrm{d}t}\nabla F(r'(t))dt\\
 & = & -\int_{a}^{b}\delta(t)\cdot\frac{\mathrm{d}}{\mathrm{d}t}\nabla F(r'(t))dt.\end{array}\label{eq:variation}\end{equation}
Thus we can obtain $\frac{\mathrm{d}}{\mathrm{d}t}\nabla F(r'(t))=0$
since $V'(0)=0$ as $V(0)\leqslant V(h)$ for any $\delta(t)$), or
by the Euler\textendash{}Lagrange equation directly, we have \begin{equation}
\frac{\mathrm{d}}{\mathrm{d}t}\nabla F(r'(t))=0.\label{eq:der}\end{equation}
Using chain rule, \prettyref{eq:der} becomes \begin{equation}
Hess(F)\frac{d^{2}r(t)}{dt^{2}}=0.\label{eq:hess}\end{equation}

On the other hand, we have \begin{equation}
\nabla F(r'(t))\frac{d^{2}r(t)}{dt^{2}}=0\label{eq:chain}\end{equation}
 by differentating $F(r'(t))=1$, and then by product rule, \prettyref{eq:hess}
and \prettyref{eq:chain}, \begin{equation}
\begin{array}{lll}
\frac{1}{2}Hess(F^{2})\frac{d^{2}r(t)}{dt^{2}} & = & F(r'(t))Hess(F)\frac{d^{2}r(t)}{dt^{2}}+(\nabla F(r'(t))^{T}\nabla F(r'(t))\frac{d^{2}r(t)}{dt^{2}}\\
 & = & Hess(F)\frac{d^{2}r(t)}{dt^{2}}\\
 & = & 0.\end{array}\label{eq:haseua}\end{equation}
Hence we get $\frac{d^{2}r(t)}{dt^{2}}=0$ because $\frac{1}{2}Hess(F^{2})$
is non-degenerated by (\prettyref{enu:hes}) in Definition \prettyref{def:A-smooth-function},
and then it implies $r(t)$, $t\in[a,b]$, is a straight line segment
connecting $p$ and $q$.
\end{proof}
Thus the space of geodesics in $(\mathbb{R}^{n},F)$ actually is the
space of affine lines, denoted by $\overline{Gr_{1}(\mathbb{R}^{n})}$.
More generally, one can define
\begin{defn}
\label{def:The-affine-Grassmannian}The affine Grassmannian $\overline{Gr_{k}(\mathbb{R}^{n})}$
is the space of affine $k$-planes in $(\mathbb{R}^{n},F)$.
\end{defn}

\subsection{Symplectic Structures on Cotangent Bundle\label{sub:Symplectic-Structures-on}}

The Minkowski space $(\mathbb{R}^{n},F)$, as a differentiable manifold,
has a canonical symplectic structure on its cotangent bundle $T^{*}\mathbb{R}^{n}$,
from which a symplectic structure on its tangent bundle $T\mathbb{R}^{n}$
can be derived as well. 

The canonical contact form $\alpha$ on $T^{*}\mathbb{R}^{n}$ is
defined as $\alpha_{\xi}(X):=\xi(\pi_{0*}X)$ for $X\in T_{\xi}T^{*}\mathbb{R}^{n}$,
where $\pi_{0}:T^{*}\mathbb{R}^{n}\to\mathbb{R}^{n}$ is the natural
projection. And then the canonical symplectic form on $T^{*}\mathbb{R}^{n}$
is defined as $\omega:=d\alpha$. 

On the other hand, we know that the dual of Minkowski metric is defined
as \begin{equation}
F^{*}(\xi):=sup\left\{ |\xi(v)|:v\in T\mathbb{R}^{n},F(v)\leqslant1\right\} \label{eq:dualnorm}\end{equation}
 for $\xi\in T^{*}\mathbb{R}^{n}$, and there is a natural correspondence
between the sphere bundle $S\mathbb{R}^{n}$and the cosphere bundle
$S\mathbb{R}^{n}=\left\{ \xi\in T^{*}\mathbb{R}^{n}:F^{*}(\xi)=1\right\} $
of the Minkowski space $(\mathbb{R}^{n},F)$. 

By the convexity and the positive homogeneity of $F$, see \cite{S1},
we can obtain $F^{*}(dF(\bar{\xi}))=1$ and $dF^{*}(dF(\bar{\xi}))=\bar{\xi}$
for any $\bar{\xi}\in S_{x}\mathbb{R}^{n}$ and $x\in\mathbb{R}^{n}$,
where $dF$ is the gradient of $F$ and similarly for $dF^{*}$. Thus
$dF$ is a diffeomorphism from $S_{x}\mathbb{R}^{n}$ to $S_{x}^{*}\mathbb{R}^{n}$,
which induces another diffeomorphism\begin{equation}
\begin{array}{c}
\varphi_{F}:S\mathbb{R}^{n}\rightarrow S^{*}\mathbb{R}^{n}\\
\varphi_{F}((x,\overline{\xi}_{x}))=(x,dF(\overline{\xi}_{x}))\end{array}\label{eq:diffeo}\end{equation}
for any $\overline{\xi}_{x}\in S_{x}\mathbb{R}^{n}$. More generally,
there is another diffeomorphism $\frac{1}{2}dF^{2}$ from $T_{x}\mathbb{R}^{n}\setminus\left\{ 0\right\} $
to $T_{x}^{*}\mathbb{R}^{n}\setminus\left\{ 0\right\} $ for any $x\in(\mathbb{R}^{n},F)$,
thus we obtain a diffeomorphism 

\begin{equation}
\begin{array}{c}
\bar{\varphi}_{F}:T\mathbb{R}^{n}\rightarrow T^{*}\mathbb{R}^{n}\\
\bar{\varphi}_{F}((x,\overline{\xi}_{x}))=(x,\frac{1}{2}dF^{2}(\overline{\xi}_{x}))\end{array}\label{eq:diffeomorege}\end{equation}
by ignoring the $0$-sections.

The diffeomorphism \prettyref{eq:diffeo} induces a $2$-form $\bar{\omega}:=\varphi_{F}^{*}\omega$
on $S\mathbb{R}^{n}$. Without loss of elegance, we can express it
more concretely. Since $T^{*}\mathbb{R}^{n}=\mathbb{R}^{n}\times\mathbb{R}^{n*}$,
for $(x,\xi)\in T^{*}\mathbb{R}^{n}$ the canonical symplectic form
$\omega$ on $T^{*}\mathbb{R}^{n}$ is actually $\omega=tr(dx\wedge d\xi)$,
here we denote $dx\wedge d\xi:=(dx_{i}\wedge d\xi_{j})_{n\times n}$
and similarly $dx\wedge d\bar{\xi}:=(dx_{i}\wedge d\bar{\xi}_{j})_{n\times n}$,
$n\times n$ matrices with $2$-forms as entries, where $\xi_{j}(\xi)=\xi(\frac{\partial}{\partial x_{j}})$,
$\bar{\xi}_{j}(\bar{\xi})=dx_{j}(\bar{\xi})$ and $F^{*}(\bar{\xi})=1$.
Then using chain rule, we can obtain

\begin{equation}
\bar{\omega}=\varphi_{F}^{*}(\omega|_{S^{*}\mathbb{R}^{n}})=Hess(F)\star dx\wedge d\bar{\xi}|_{S\mathbb{R}^{n}},\label{eq:matri}\end{equation}
where $\star$ is the Frobenius inner product which is the sum of
the entries of the entrywise product of two matrices.

\subsection{Gelfand Transform }

Gelfand transform on a double fibration as a generalization of Radon
transform plays an important role in making use of the symplectic
form of Section \prettyref{sub:Symplectic-Structures-on} in integral
geometry of Minkowski space.
\begin{defn}
Let $M\overset{\pi_{1}}{\leftarrow}\mathcal{F}\overset{\pi_{2}}{\rightarrow}\Gamma$
be double fibration where $M$ and $\Gamma$ are two manifolds, $\pi_{1}:\mathcal{F}\rightarrow M$
and $\pi_{2}:\mathcal{F}\rightarrow\Gamma$ are two fibre bundles,
and $\pi_{1}\times\pi_{2}:\mathcal{F}\rightarrow M\times\Gamma$ is
an submersion. Let $\Phi$ be a density on $\Gamma$, then the Gelfand
transform of $\Phi$ is defined as $GT(\Phi):=\pi_{1*}\pi_{2}^{*}\Phi$.
In the case $\Phi$ is a differential form and the fibres are oriented,
then we also have a well-defined Gelfand transform $GT(\Phi):=\pi_{1*}\pi_{2}^{*}\Phi$,
noting that the pushforward of a form is the integral of contracted
form over the fibre.

To make it clear, let's see how the degree of a density or form changes
by the transform. Suppose $\Phi$ is a density or form of degree $m$
on $\Gamma$ and the dimension of fibre $\pi_{1}$is $q$, then $\pi_{2}^{*}\Phi$
has degree $m$, and then $GT(\Phi)=\pi_{1*}\pi_{2}^{*}\Phi=\intop_{\pi_{1}^{-1}(x)}\pi_{2}^{*}\Phi$
for $x\in M$ has degree $m-q$.

An application of Gelfand transforms in integral geometry is the following
fundamental theorem \cite{AF1}, whose proof is quite simple.\end{defn}
\begin{thm}
\label{thm:fundamental}Suppose $M_{\gamma}:=\pi_{1}(\pi_{2}^{-1}(\gamma))$
are smooth submanifolds of $M$ for $\gamma\in\Gamma$, $\overline{M}\subset M$
is a immersed submanifold, and $\Phi$ is a top degree density on
$\Gamma$. Then \begin{equation}
\int_{\Gamma}\#(\overline{M}\cap M_{\gamma})\Phi(\gamma)=\int_{\overline{M}}GT(\Phi).\end{equation}
\end{thm}
\begin{proof}
Working on the transitions of measures on manifolds and the transformations
of intersection numbers, we have\begin{equation}
\begin{array}{lllll}
\int_{\overline{M}}GT(\Phi) & = & \int_{\overline{M}}\pi_{1*}\pi_{2}^{*}\Phi & = & \int_{\pi_{1}^{-1}(\overline{M})}\pi_{2}^{*}\Phi\\
 &  &  & = & \int_{\Gamma}\#(\pi_{2}^{-1}(\gamma)\cap\pi_{1}^{-1}(\overline{M}))\Phi(\gamma)\\
 &  &  & = & \int_{\Gamma}\#(\overline{M}\cap M_{\gamma})\Phi(\gamma).\end{array}\end{equation}

\end{proof}

\section{\label{sec:The-Symplectic-Structure}The Symplectic Structure on
the Space of Geodesics}

The symplectic structure on the space of geodesics in a Minkowski
space is induced naturally from the canonical symplectic structure
on its cotangent bundle. 

The process of construction of symplectic form on $\overline{Gr_{1}(\mathbb{R}^{n})}$
in Minkowski space $(\mathbb{R}^{n},F)$ is based on the following
diagram\begin{equation}
\begin{array}{ccccc}
S\mathbb{R}^{n} & \overset{\overset{\varphi_{F}}{\simeq}}{\rightarrow} & S^{*}\mathbb{R}^{n} & \overset{i}{\hookrightarrow} & T^{*}\mathbb{R}^{n}\\
\downarrow p\\
\overline{Gr_{1}(\mathbb{R}^{n})}\end{array}\label{eq:prodiag}\end{equation}
where $p$ is the projection from $S\mathbb{R}^{n}$ onto $\overline{Gr_{1}(\mathbb{R}^{n})}$
defined by 

\begin{equation}
p((x,\bar{\xi})):=l(x,\bar{\xi}),\label{eq:linepro}\end{equation}
where $l(x,\bar{\xi})$ is the line passing through $x$ with direction
$\bar{\xi}$. 

Consider the geodesic vector field $\overline{\mathcal{X}}(\overline{\xi}_{x}):=(\overline{\xi}_{x},0)$
on $T\mathbb{R}^{n}$ for any $\overline{\xi}_{x}\in S\mathbb{R}^{n}$,
$\varphi_{F}$ in \prettyref{eq:diffeo} induces another vector field
$\mathcal{X}:=d\varphi_{F}(\overline{\mathcal{X}})$ on $T^{*}\mathbb{R}^{n}$
with \begin{equation}
\mathcal{X}(dF(\overline{\xi}_{x}))=(d\varphi_{F}(\overline{\mathcal{X}})(\varphi_{F}(\overline{\xi}_{x}))=(\overline{\xi}_{x},0)\label{eq:dualvecor}\end{equation}
for $\overline{\xi}_{x}\in S\mathbb{R}^{n}$. 

We have the following vanishing property about $\mathcal{X}$ and
$\omega$ on $S^{*}\mathbb{R}^{n}$. 
\begin{lem}
\label{lem:vani}$i_{\mathcal{X}}\omega=0$ on $S^{*}\mathbb{R}^{n}$
. \end{lem}
\begin{proof}
Noting that $\omega(X,Y)=\langle X_{1},Y_{2}\rangle-\langle Y_{1},X_{2}\rangle$
for any $X=(X_{1},X_{2})$ and $Y=(Y_{1},Y_{2})$ in $T_{\xi_{x}}S^{*}\mathbb{R}^{n}\subset T_{\xi_{x}}T^{*}\mathbb{R}^{n}$
because $T^{*}\mathbb{R}^{n}\cong\mathbb{R}^{n}\times\mathbb{R}^{n*}$,
where the inner product is the dual space action, by \prettyref{eq:dualvecor}
we have \begin{equation}
\omega_{\xi_{x}}(\mathcal{X},Y)=\langle\overline{\xi}_{x},Y_{2}\rangle=\langle dF^{*}(\xi_{x}),Y_{2}\rangle=0\end{equation}
because $Y_{2}\in T_{\xi_{x}}S^{*}\mathbb{R}^{n}$ is {}``normal''
to $dF^{*}(\xi_{x})$, precisely, that can be obtained by differentiating
$F^{*}(\xi_{x})=1$ and noting $Y_{2}\in T_{\xi_{x}}S^{*}\mathbb{R}^{n}$.
\end{proof}
Furthermore, the Lie derivative of $\omega$ along geodesic vector
field $\mathcal{X}$ is \begin{equation}
\mathcal{L}_{\mathcal{X}}\omega=di_{\mathcal{X}}\omega+i_{\mathcal{X}}d\omega=0\label{eq:lie}\end{equation}
by \prettyref{lem:vani}. Then \prettyref{eq:lie} implies $(\varphi_{F})^{*}\omega|_{S^{*}\mathbb{R}^{n}}$
is invariant under $\mathcal{\bar{X}}$. 

Based on the invariance of $\omega$ we can construct a symplectic
structure on $\overline{Gr_{1}(\mathbb{R}^{n})}$. However, in order
to do that, we need to give a manifold structure for $\overline{Gr_{1}(\mathbb{R}^{n})}$
first.

In fact, we can build a bijection $\psi$ between $\overline{Gr_{1}(\mathbb{R}^{n})}$
and $TS_{F}^{n-1}$,where $S_{F}^{n-1}$ is the unit sphere in $(\mathbb{R}^{n},F)$.
For any $l(x,\bar{\xi})\in\overline{Gr_{1}(\mathbb{R}^{n})}$, let
$\bar{\eta}$ be the tangent vector pointing at $l(x.\bar{\xi})\cap T_{\bar{\xi}}S_{F}^{n-1}$,
in fact, $\bar{\eta}=x-dF(\bar{\xi})(x)\bar{\xi}\in T_{\bar{\xi}}S_{F}^{n-1}$,
see \prettyref{fig:isom}, and one can define \begin{equation}
\psi(l(x,\bar{\xi})):=(\bar{\xi,}\bar{\eta})=(\bar{\xi},x-dF(\bar{\xi})(x)\bar{\xi}).\label{eq:bij}\end{equation}
Thus we have a homeomorphism $\psi$ from $\overline{Gr_{1}(\mathbb{R}^{n})}$
to $TS_{F}^{n-1}$, and then the manifold structure on $TS_{F}^{n-1}$
provides one for $\overline{Gr_{1}(\mathbb{R}^{n})}$.

\begin{figure}
\includegraphics[scale=0.6]{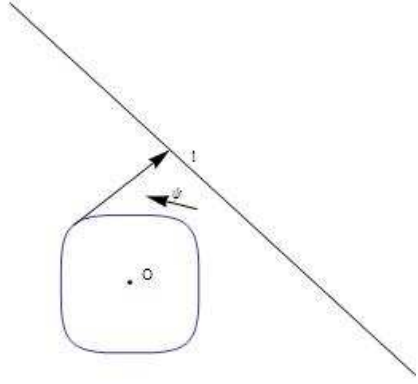}

\caption{\label{fig:isom} $\overline{Gr_{1}(\mathbb{R}^{n})}$ Diffeomorphic
to $TS_{F}^{n-1}$}

\end{figure}

Let us again consider the projection \prettyref{eq:linepro} with
the manifold structure on $\overline{Gr_{1}(\mathbb{R}^{n})}$, and
then we can obtain the following lemma
\begin{lem}
\label{lem:kernel} $\mathcal{\bar{X}}$ is in the kernel of $dp$,
in other words, $p_{*}(\mathcal{\bar{X}})=0$.\end{lem}
\begin{proof}
Using the basic equality \begin{equation}
dF(\bar{\xi})(\bar{\xi})=F(\bar{\xi})=1\label{eq:one}\end{equation}
obtained by the positive homogeneity of $F$ for any $\bar{\xi}\in S_{x}\mathbb{R}^{n}$,
we have \begin{equation}
\begin{array}{lllll}
p_{*}(\mathcal{\bar{X}}) & = & dp((\bar{\xi},0)) & = & d(\bar{\xi},x-dF(\bar{\xi})(x)\bar{\xi})((\bar{\xi},0))\\
 &  &  & = & \bar{\xi}-dF(\bar{\xi})(\bar{\xi})\bar{\xi}\\
 &  &  & = & (1-dF(\bar{\xi})(\bar{\xi}))\bar{\xi}\\
 &  &  & = & 0.\end{array}\end{equation}

\end{proof}
One can compute the rank of the Jacobian of $p$ which is $2n-2$,
that implies $dim(dp|_{\bar{\xi}_{x}})=1$ and then \begin{equation}
ker(dp|_{\bar{\xi}_{x}})=span(\mathcal{\bar{X}}(\bar{\xi}_{x}))\label{eq:span}\end{equation}
by \prettyref{lem:kernel}. 

Now we can obtain the following theorem 
\begin{thm}
\label{thm:There-exists-a}There exists a symplectic form $\omega_{0}$
on $\overline{Gr_{1}(\mathbb{R}^{n})}$, such that $p^{*}\omega_{0}=\bar{\omega}=(\varphi_{F})^{*}\omega|_{S^{*}\mathbb{R}^{n}}$.\end{thm}
\begin{proof}
By \prettyref{eq:lie}, \prettyref{eq:span} and \prettyref{lem:vani},
we know that $\bar{\omega}_{\xi_{x}}(X,Y)$ is independent of the
choices of preimages under the pushforward induced by projection $p$.
Thus we have a well-defined two form $\omega_{0}$ on $\overline{Gr_{1}(\mathbb{R}^{n})}$,
\begin{equation}
\omega_{0_{p(\xi_{x})}}(\tilde{X},\tilde{Y}):=\bar{\omega}_{\xi_{x}}(X,Y),\end{equation}
where $(p_{*})_{\xi{}_{x}}(X)=\tilde{X}$ and $(p_{*})_{\xi{}_{x}}(Y)=\tilde{Y}$,
such that \begin{equation}
p^{*}\omega_{0}=\bar{\omega}=(\varphi_{F})^{*}i^{*}\omega.\label{eq:imbeddingindu}\end{equation}

\end{proof}
That finishes the construction of symplectic structure on the space
of geodesics in Minkowski space.

On the other hand, since $T^{*}S_{F}^{n-1}$ as a cotangent bundle
on Riemannian manifold $S_{F}^{n-1}$ has a canonical symplectic structure
denoted as $\tilde{\omega}$, and we have a canonical diffeomorphism\begin{equation}
\begin{array}{c}
\tilde{\varphi}_{F}:TS_{F}^{n-1}\rightarrow T^{*}S_{F}^{n-1}\\
\tilde{\varphi}_{F}(\bar{\eta}_{\bar{\xi}})=\langle\bar{\eta}_{\bar{\xi}},\cdot\rangle_{g_{F}},\end{array}\label{eq:phimap}\end{equation}
in which $g_{F}$ is the Riemannian metric on $S_{F}^{n-1}$, which
is actually the bilinear form \begin{equation}
\langle\bar{u},\bar{v}\rangle_{g_{F}}:=\frac{\partial^{2}}{\partial s\partial t}F(\bar{\xi}+s\bar{u}+t\bar{v})|_{s=t=0}\label{eq:inner}\end{equation}
for any $\bar{u},\bar{v}\in T_{\bar{\xi}}S_{F}^{n-1}$, see \cite{CS},
and then $\tilde{\varphi}_{F}^{*}\tilde{\omega}$ is the the symplectic
form induced on $TS_{F}^{n-1}\cong\overline{Gr_{1}(\mathbb{R}^{n})}$.
Also, we have another symplectic form $\omega_{0}$ on $\overline{Gr_{1}(\mathbb{R}^{n})}$
from \prettyref{thm:There-exists-a}. A natural question is whether
the two symplectic structures on $\overline{Gr_{1}(\mathbb{R}^{n})}$
are the same, the answer is yes, see the following theorem 
\begin{thm}
$\omega_{0}=\tilde{\varphi}_{F}^{*}\bar{\omega}$.
\end{thm}
Let us first draw a diagram for this theorem by combining \prettyref{eq:prodiag}\begin{equation}
\begin{array}{ccccc}
S\mathbb{R}^{n} & \overset{\overset{\varphi_{F}}{\simeq}}{\rightarrow} & S^{*}\mathbb{R}^{n} & \overset{i}{\hookrightarrow} & T^{*}\mathbb{R}^{n}\\
\downarrow p\\
\overline{Gr_{1}(\mathbb{R}^{n})}\overset{\psi}{\simeq}TS_{F}^{n-1} & \overset{\overset{\tilde{\varphi}_{F}}{\simeq}}{\rightarrow} & T^{*}S_{F}^{n-1}.\end{array}\end{equation}

\begin{proof}
First, differentiating \prettyref{eq:one} and using chain rule, one
can get \begin{equation}
Hess(F)\star\bar{\xi}d\bar{\xi}|_{S\mathbb{R}^{n}}=0\label{eq:secdir}\end{equation}
in which $\bar{\xi}d\bar{\xi}:=(\bar{\xi}_{i}d\bar{\xi}_{j})_{n\times n}$
is a matrix and $\star$ is the Frobenius inner product of matrices.

Next, the canonical symplectic form $\tilde{\omega}$ on $T^{*}S_{F}^{n-1}$,
$\tilde{\omega}=\omega|_{T^{*}S_{F}^{n-1}}$ in which $\omega$ is
the canonical symplectic form on the cotangent bundle $T^{*}\mathbb{R}^{n}$.
Thus, from \prettyref{eq:phimap} and \prettyref{eq:inner}, one can
obtain that \begin{equation}
\tilde{\varphi}_{F}^{*}\tilde{\omega}=Hess(F)\star d\bar{\eta}\wedge d\bar{\xi}|_{TS_{F}^{n-1}},\label{eq:mapbetwe}\end{equation}
here $d\bar{\xi}\wedge d\bar{\eta}$ is a matrix with $2$-form entries
and $\star$ is the Frobenius inner product of matrices. 

Therefore, by plugging \prettyref{eq:bij} into \prettyref{eq:mapbetwe}
and using \prettyref{eq:secdir}, we obtain \begin{equation}
\begin{array}{lll}
p^{*}\tilde{\varphi}_{F}^{*}\tilde{\omega} & = & Hess(F)\star d\bar{\eta}\wedge d\bar{\xi}|_{TS_{F}^{n-1}}\\
 & = & Hess(F)\star d(x-dF(\bar{\xi})(x)\bar{\xi})\wedge d\bar{\xi}|_{S\mathbb{R}^{n}}\\
 & = & Hess(F)\star dx\wedge d\bar{\xi}|_{S\mathbb{R}^{n}}-d(dF(\bar{\xi})(x))\wedge Hess(F)\star\bar{\xi}d\bar{\xi}|_{S\mathbb{R}^{n}}\\
 & = & Hess(F)\star dx\wedge d\bar{\xi}|_{S\mathbb{R}^{n}}\\
 & = & \bar{\omega},\end{array}\end{equation}
which by \prettyref{thm:There-exists-a} implies the claim.
\end{proof}
At the end to this section, we make a remark on the symplectic structure
on $\overline{Gr_{1}(\mathbb{R}^{n})}$.
\begin{rem}
From \prettyref{eq:matri} we see the symplectic structure $\bar{\omega}$
on $T\mathbb{R}^{n}$ relies on the Minkowski metric $F$, then we
know, by the above construction, the symplectic structure on $\overline{Gr_{1}(\mathbb{R}^{n})}$
depends on the Minkowski metric $F$ as well. Let us see the following
example of Minkowski plane with $p$-norm as a Minkowski metric.\end{rem}
\begin{example}
Given a Minkowski plane by $\left(\mathbb{R}^{2},||\cdot||_{p}\right)$,
$1<p<\infty$, where $||(\alpha,\beta)||_{p}=(|\alpha|^{p}+|\beta|^{p})^{1/p}$
and the dual norm is $||\cdot||_{\frac{p}{p-1}}$ we can obtain the
symplectic form $\omega$ on $\overline{Gr_{1}(\mathbf{\mathbb{R}^{2})}}$,
the space of affine lines in $\left(\mathbb{R}^{2},||\cdot||_{p}\right)$,
by following the general construction above. 

By \prettyref{eq:matri} and \prettyref{thm:There-exists-a}, we have
\begin{equation}
p^{*}\omega_{0}=(p-1)\alpha^{p-2}dx\wedge d\alpha+(p-1)\beta^{p-2}dy\wedge d\beta,\label{eq:symponsn}\end{equation}
for $((x,y),(\alpha,\beta))\in S\mathbb{R}^{2}$.

Since $\overline{Gr_{1}(\mathbf{\mathbb{R}^{2})}}$ is a $2$-dimensional
manifold, we can parametrize affine lines in $\overline{Gr_{1}(\mathbf{\mathbb{R}^{2})}}$
with two variables in a natural way. For any straight line $l$ passing
through $(x,y)$ with direction $(\alpha,\beta)$ of unit $p$-norm,
let $(-\Theta,\Omega)$ be the unit vector in $p$-norm such that
$l$ is tangent to the Minkowski sphere $S(r)$ of radius $r$ at
$(-r\Theta,r\Omega)$, here we can call $r$ the {}``$p$-norm distance''
of $l$ to the origin, see \prettyref{fig:pdist}. Thus we can denote
the line by $l(r,\Theta)$.

\begin{figure}
\includegraphics[scale=0.5]{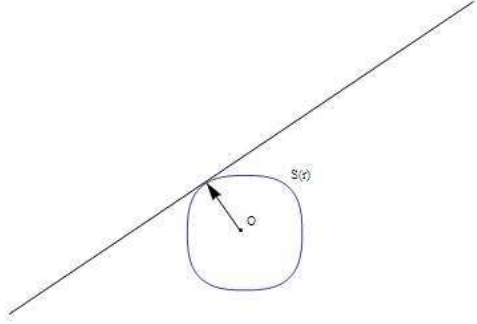}

\caption{\label{fig:pdist}{}``$p$-norm distance'' $r$}

\end{figure}

We have the following theorem about the symplectic structure on $\overline{Gr_{1}(\mathbf{\mathbb{R}^{2})}}$
by the above parametrization.\end{example}
\begin{thm}
The symplectic structure on $\overline{Gr_{1}(\mathbf{\mathbb{R}^{2})}}$
is \begin{equation}
\omega_{0}=\frac{(p-1)^{2}\Theta^{p(p-2)}\Omega^{p^{2}-3p+1}}{||(\Theta,\Omega)||_{p(p-1)}^{(p-1)(2p-1)}}dr\wedge d\Theta.\label{eq:psymp}\end{equation}
\end{thm}
\begin{proof}
For a line $l$ passing through $(x,y)$ with direction $(\alpha,\beta)$
of unit $p$-norm, the {}``$p$-norm distance''\begin{equation}
r=-\Theta^{p-1}x+\Omega^{p-1}y\label{eq:dist}\end{equation}
 and\begin{equation}
(-\Theta,\Omega)=(-\frac{\beta^{\frac{1}{p-1}}}{(\alpha^{\frac{p}{p-1}}+\beta^{\frac{p}{p-1}})^{\frac{1}{p}}},\frac{\alpha^{\frac{1}{p-1}}}{(\alpha^{\frac{p}{p-1}}+\beta^{\frac{p}{p-1}})^{\frac{1}{p}}}).\label{eq:vect}\end{equation}

In order to express $\omega_{0}$ in terms of $r$ and $\Theta$,
at first we use \prettyref{eq:dist} and \prettyref{eq:vect} to compute\begin{equation}
\begin{array}{lll}
dr\wedge d\Theta & = & (-\Theta^{p-1}dx+\Omega^{p-1}dy)\wedge d\Theta\\
 & = & -\Theta^{p-1}dx\wedge d(\frac{\beta^{\frac{1}{p-1}}}{(\alpha^{\frac{p}{p-1}}+\beta^{\frac{p}{p-1}})^{\frac{1}{p}}})+\Omega^{p-1}dy\wedge d(\frac{\beta^{\frac{1}{p-1}}}{(\alpha^{\frac{p}{p-1}}+\beta^{\frac{p}{p-1}})^{\frac{1}{p}}})\\
 & = & -\Theta^{p-1}dx\wedge d(\frac{1}{((\frac{\alpha}{\beta})^{\frac{p}{p-1}}+1)^{\frac{1}{p}}})+\Omega^{p-1}dy\wedge d(\frac{1}{((\frac{\alpha}{\beta})^{\frac{p}{p-1}}+1)^{\frac{1}{p}}})\\
 & = & -\Theta^{p-1}dx\wedge(-\frac{1}{p})((\frac{\alpha}{\beta})^{\frac{p}{p-1}}+1)^{-\frac{p+1}{p}}\frac{p}{p-1}(\frac{\alpha}{\beta})^{\frac{1}{p-1}}\frac{\beta d\alpha-\alpha d\beta}{\beta^{2}}\\
 &  & \,+\Omega^{p-1}dy\wedge(-\frac{1}{p})((\frac{\alpha}{\beta})^{\frac{p}{p-1}}+1)^{-\frac{p+1}{p}}\frac{p}{p-1}(\frac{\alpha}{\beta})^{\frac{1}{p-1}}\frac{\beta d\alpha-\alpha d\beta}{\beta^{2}}\\
 & = & (\frac{1}{p-1})(\frac{\alpha}{\beta})^{\frac{1}{p-1}}((\frac{\alpha}{\beta})^{\frac{p}{p-1}}+1)^{-\frac{p+1}{p}}(\Theta^{p-1}dx\wedge(\frac{1}{\beta}d\alpha-\frac{\alpha}{\beta^{2}}d\beta)\\
 &  & \qquad-\Omega^{p-1}dy\wedge(\frac{1}{\beta}d\alpha-\frac{\alpha}{\beta^{2}}d\beta)\\
 & = & (\frac{1}{p-1})(\frac{\alpha}{\beta})^{\frac{1}{p-1}}((\frac{\alpha}{\beta})^{\frac{p}{p-1}}+1)^{-\frac{p+1}{p}}(\Theta^{p-1}(\frac{1}{\beta}+\frac{\alpha^{p}}{\beta^{p+1}})dx\wedge d\alpha\\
 &  & \qquad-\Omega^{p-1}(-\frac{1}{\beta}\frac{\beta^{p-1}}{\alpha^{p-1}}-\frac{\alpha}{\beta^{2}})dy\wedge d\beta)\\
 & = & (\frac{1}{p-1})(\frac{\alpha}{\beta})^{\frac{1}{p-1}}((\frac{\alpha}{\beta})^{\frac{p}{p-1}}+1)^{-\frac{p+1}{p}}(\frac{\Theta^{p-1}}{\beta^{p+1}}dx\wedge d\alpha+\frac{\Omega^{p-1}}{\alpha^{p-1}\beta^{2}}dy\wedge d\beta)\\
 & = & \frac{1}{(p-1)^{2}}(\frac{\alpha}{\beta})^{\frac{1}{p-1}}((\frac{\alpha}{\beta})^{\frac{p}{p-1}}+1)^{-\frac{p+1}{p}}(\frac{\Theta^{p-1}}{\beta^{p+1}\alpha^{p-2}}(p-1)\alpha^{p-2}dx\wedge d\alpha\\
 &  & \qquad+\frac{\Omega^{p-1}}{\alpha^{p-1}\beta^{p}}(p-1)\beta^{p-2}dy\wedge d\beta).\end{array}\label{eq:firstcp}\end{equation}
Indeed,\begin{equation}
\frac{\Theta^{p-1}}{\beta^{p+1}\alpha^{p-2}}=\frac{\Omega^{p-1}}{\alpha^{p-1}\beta^{p}}\label{eq:propomeley}\end{equation}
since \begin{equation}
\frac{\Theta^{p-1}}{\Omega^{p-1}}=\frac{\beta}{\alpha}\label{eq:propotion}\end{equation}
 by \prettyref{eq:vect}. 

Therefore, using \prettyref{eq:propomeley}, \prettyref{eq:propotion}
and $||(\Theta,\Omega)||_{p}=1$, we have\begin{equation}
\begin{array}{lll}
dr\wedge d\Theta & = & \frac{1}{(p-1)^{2}}(\frac{\alpha}{\beta})^{\frac{1}{p-1}}((\frac{\alpha}{\beta})^{\frac{p}{p-1}}+1)^{-\frac{p+1}{p}}\frac{\Theta^{p-1}}{\beta^{p+1}\alpha^{p-2}}\\
 &  & \qquad((p-1)\alpha^{p-2}dx\wedge d\alpha+(p-1)\beta^{p-2}dy\wedge d\beta)\\
 & = & \frac{1}{(p-1)^{2}}(\frac{\alpha}{\beta})^{\frac{1}{p-1}}((\frac{\alpha}{\beta})^{\frac{p}{p-1}}+1)^{-\frac{p+1}{p}}\frac{\Theta^{p-1}}{\beta^{p+1}\alpha^{p-2}}\omega_{0}\\
 & = & \frac{1}{(p-1)^{2}}\frac{(\frac{\alpha}{\beta})^{\frac{1}{p-1}}\Theta^{p-1}}{((\frac{\alpha}{\beta})^{\frac{p}{p-1}}+1)^{\frac{p+1}{p}}\beta^{p+1}\alpha^{p-2}}\omega_{0}\\
 & = & \frac{1}{(p-1)^{2}}\frac{\frac{\Omega}{\Theta}\Theta^{p-1}}{((\frac{\Omega}{\Theta})^{p}+1)^{\frac{p+1}{p}}(\frac{\Theta^{p-1}}{||(\Theta^{p-1},\Omega^{p-1})||_{p}})^{p+1}(\frac{\Omega^{p-1}}{||(\Theta^{p-1},\Omega^{p-1})||_{p}})^{p-2}}\omega_{0}\\
 & = & \frac{1}{(p-1)^{2}}\frac{\Omega\Theta^{2p-1}}{(\frac{\Theta^{p-1}}{||(\Theta^{p-1},\Omega^{p-1})||_{p}})^{p+1}(\frac{\Omega^{p-1}}{||(\Theta^{p-1},\Omega^{p-1})||_{p}})^{p-2}}\omega_{0}\\
 & = & \frac{1}{(p-1)^{2}}\frac{\Omega\Theta^{2p-1}||(\Theta^{p-1},\Omega^{p-1})||_{p}^{2p-1}}{\Theta^{(p-1)(p+1)}\Omega^{(p-1)(p-2)}}\omega_{0}\\
 & = & \frac{1}{(p-1)^{2}}\frac{||(\Theta^{p-1},\Omega^{p-1})||_{p}^{2p-1}}{\Theta^{p(p-2)}\Omega^{p^{2}-3p+1}}\omega_{0}\\
 & = & \frac{||(\Theta,\Omega)||_{p(p-1)}^{(p-1)(2p-1)}}{(p-1)^{2}\Theta^{p(p-2)}\Omega^{p^{2}-3p+1}}\omega_{0},\end{array}\end{equation}
Thus we have shown \begin{equation}
dr\wedge d\Theta=\frac{||(\Theta,\Omega)||_{p(p-1)}^{(p-1)(2p-1)}}{(p-1)^{2}\Theta^{p(p-2)}\Omega^{p^{2}-3p+1}}\omega_{0},\end{equation}
which implies \prettyref{eq:psymp} in the claim.
\end{proof}
So from \prettyref{eq:symponsn} and \prettyref{eq:psymp} we see
the symplectic structure on $\overline{Gr_{1}(\mathbf{\mathbb{R}^{2})}}$
is determined by the Minkowski metric $||\cdot||_{p}$ on $\mathbb{R}^{2}$.

\section{Integral Geometry on Length in Minkowski Space}

The length of a straight line segment in $(\mathbb{R}^{2},F)$ can
be obtained by integrating the canonical contact form $\alpha$ introduced
in Section \prettyref{sub:Symplectic-Structures-on}. For any $x,y\in\mathbb{R}^{2}$,
let $\overrightarrow{xy}$ be the vector from $x$ to $y$, and \begin{equation}
c(t):=(x+\frac{t}{F(\overrightarrow{xy})}(y-x),dF(\frac{\overrightarrow{xy}}{F(\overrightarrow{xy})})),\, t\in[0,F(\overrightarrow{xy})]\end{equation}
 be a straight line segment in $T^{*}\mathbb{R}^{2}$. By the positive
homogeneity of F, one can get the useful fact that \begin{equation}
dF(\frac{\overrightarrow{xy}}{F(\overrightarrow{xy})})(\frac{\overrightarrow{xy}}{F(\overrightarrow{xy})})=F(\frac{\overrightarrow{xy}}{F(\overrightarrow{xy})})=1.\end{equation}
Therefore, \begin{equation}
\begin{array}{lllcccc}
\int_{c}\alpha & = & \int_{0}^{F(\overrightarrow{xy})}dF(\frac{\overrightarrow{xy}}{F(\overrightarrow{xy})})(\frac{\overrightarrow{xy}}{F(\overrightarrow{xy})})dt & = & F(\overrightarrow{xy}) & = & \mathbf{L}(\overline{xy}),\end{array}\label{eq:length}\end{equation}
where $\mathbf{L}(\overline{xy})$ is the length of $\overline{xy}$. 

Here let us introduce a general definition in integral geometry first. 
\begin{defn}
A Crofton measure $\phi$ for a degree $k$ measure $\Phi$ on $(\mathbb{R}^{n},F)$
is a measure on $\overline{Gr_{n-k}(\mathbb{R}^{n})}$ (Definition
\prettyref{def:The-affine-Grassmannian}), such that it satisfies
the Crofton-type formula\begin{equation}
\Phi(M)=\int_{P\in\overline{Gr_{n-k}(\mathbb{R}^{n})}}\#(M\cap P)\Phi(P)\end{equation}
for any compact convex subset $M$ $(\mathbb{R}^{n},F)$.
\end{defn}
Furthermore, we have the following 
\begin{prop}
\label{pro:The-Crofton-measure}The Crofton measure on $\overline{Gr_{1}(\mathbb{R}^{2})}$
for the length is $|\omega_{0}|$.
\end{prop}
Our treatment of applying Stokes' theorem here is primarily based
on \cite{AD}.
\begin{proof}
From \prettyref{sec:The-Symplectic-Structure}, we know $\overline{Gr_{1}(\mathbb{R}^{2})}\overset{\overset{\psi}{\simeq}}{\rightarrow}TS_{F}$
which is a cylinder, and it has a symplectic form $\omega_{0}$ as
$TS_{F}$ embedded in $T^{*}\mathbb{R}^{2}$. 

Let $S:=\left\{ l\in\overline{Gr_{1}(\mathbb{R}^{2})}:l\cap\overset{\circ}{\overline{xy}}\neq\phi\right\} $,
$C_{x}$ and $C_{y}$ be the family of oriented lines passing through
$x$ and $y$ respectively, then $C_{x}\cap C_{y}=\left\{ l_{xy}^{+},l_{xy}^{-}\right\} $
that are the two oriented lines connecting $x$ and $y$, and $\partial S=C_{x}\cup C_{y}$. 

Let $R:=\left\{ \xi\in S\mathbb{R}^{2}\subset T\mathbb{R}^{2}:l((x+t(y-x),dF(\xi))\cap\overline{xy}\neq\phi\right\} $,
where $l((x+t(y-x),dF(\xi))$ is the line passing through $x+t(y-x)$
with direction $\xi$, then $p(R)=S$, where $p$ is the natural projection
from $S\mathbb{R}^{2}$ to $\overline{Gr_{1}(\mathbb{R}^{2})}$. 

Additionally, let $C_{x}^{'}=\left\{ \xi_{x}:\xi_{x}\in S_{x}\mathbb{R}^{2}\right\} ,$
$C_{y}^{'}=\left\{ \xi_{y}:\xi_{y}\in S_{y}\mathbb{R}^{2}\right\} ,$
$l_{xy}^{'+}=\frac{\overrightarrow{xy}}{F(\overrightarrow{xy})}\in S_{x}\mathbb{R}^{2},$
and $l_{xy}^{'-}=-\frac{\overrightarrow{xy}}{F(\overrightarrow{xy})}\in S_{y}\mathbb{R}^{2},$
then $p$ maps $C_{x}^{'}$, $C_{y}^{'}$, $l_{xy}^{'+}$ and $l_{xy}^{'-}$
to $C_{x}$, $C_{y}$, $l_{xy}^{+}$ and $l_{xy}^{-}$ respectively,
see \prettyref{fig:cyl}.

\begin{figure}
\includegraphics[scale=0.55]{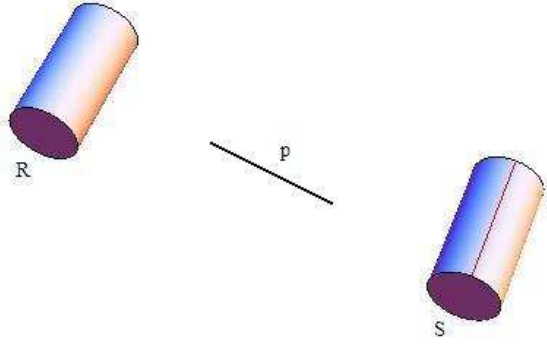}

\caption{\label{fig:cyl}From $S\mathbb{R}^{2}$ to $\overline{Gr_{1}(\mathbb{R}^{2})}$}

\end{figure}

Applying Stokes' theorem to the two regions individually, using the
fact that $\int_{C'}\alpha=0$ because of the fixed base points for
any $C'\subset C_{x}^{'}\mbox{\, or\,}C_{y}^{'}$ , and combining
with \prettyref{eq:length}, we obtain\begin{equation}
\begin{array}{lllll}
\int_{S}|\omega_{0}| & = & \int_{p(R)}|\omega_{0}| & = & \int_{R}|p^{*}\omega_{0}|\\
 &  &  & = & \int_{R}|\omega|\\
 &  &  & = & 2\int_{l_{xy}^{'+}\cup l_{xy}^{'-}}\alpha\\
 &  &  & = & 4\mathbf{L}(\overline{xy}).\end{array}\end{equation}

Therefore, for any rectifiable curve $\gamma$ in $(\mathbb{R}^{2},F)$,
the length of $\gamma$,\begin{equation}
\mathbf{L}(\gamma)=\frac{1}{4}\int_{l\in\overline{Gr_{1}(\mathbb{R}^{2})}}\#(\gamma\cap l)|\omega_{0}|,\end{equation}
which is the desired claim.\end{proof}
\begin{rem}
The proof above can be applied to $\mathbb{R}^{2}$ with projective
Finsler metric, in which geodesics are straight lines. Furthermore,
for $\mathbb{R}^{n}$ with a projective Finsler metric $F$, we choose
a plane $P\subset\mathbb{R}^{n}$ containing $\overline{xy}$ for
any $x,y\in\mathbb{R}^{n}$, then\begin{equation}
\mathbf{L}(\overline{xy})=\frac{1}{4}\int_{l\in\overline{Gr_{1}(P)}}\#(\overline{xy}\cap l)|\omega_{0}|.\end{equation}

\end{rem}

\section{Volume of Hypersurfaces}

A standard definition of Holmes-Thompson volume in Minkowski space
$(\mathbb{R}^{n},F)$ is given and its importance in Finsler geometry
and integral geometry is illustrated in \cite{AB}. 

The Holmes-Thompson volumes are defined as follows.
\begin{defn}
Let $N$ be a $k$-dimensional manifold and \begin{equation}
D^{*}N:=\left\{ \xi_{x}\in T^{*}N:F^{*}(\xi_{x})\leqslant1\right\} ,\end{equation}
where $F^{*}$ is the dual norm in \prettyref{eq:dualnorm}, be the
codisc bundle of $N$ , then the $k$-th Holmes-Thompson volume is
defined as \begin{equation}
vol_{k}(N):=\frac{1}{\epsilon_{k}}\int_{D^{*}N}|\omega^{k}|,\end{equation}
where $\epsilon_{k}$ is the Euclidean volume of $k$-dimensional
Euclidean ball and $\omega$ is the canonical symplectic form on the
cotangent bundle of $N$.

Let $\Lambda\in\overline{Gr_{k}(\mathbb{R}^{n})}$ for some $k\leq n$,
$\omega_{0}$ and $\hat{\omega}_{0}$ are the natural symplectic forms
on $\overline{Gr_{1}(\mathbb{R}^{n})}$ and $\overline{Gr_{1}(\Lambda)}$
constructed in the way described in \prettyref{sec:The-Symplectic-Structure}.
The relation between $\omega_{0}$ and $\hat{\omega}_{0}$ is shown
in the following\end{defn}
\begin{lem}
$i^{*}\omega_{0}=\omega$ for $i:\overline{Gr_{1}(\Lambda)}\hookrightarrow\overline{Gr_{1}(\mathbb{R}^{n})}$.\end{lem}
\begin{proof}
First consider the diagram \begin{equation}
S^{*}\Lambda\overset{\overset{\varphi_{F^{*}}}{\simeq}}{\rightarrow}S\Lambda\overset{\hat{i}}{\hookrightarrow}S\mathbb{R}^{n}\overset{\overset{\varphi_{F}}{\simeq}}{\rightarrow}S^{*}\mathbb{R}^{n}.\label{eq:embdiag}\end{equation}
We have a canonical contact form $\hat{\alpha}_{\xi}(X):=\xi(\hat{\pi}_{0*}X)$
for $X\in T_{\xi}S^{*}\Lambda$ on $S^{*}\Lambda$ in diagram \prettyref{eq:embdiag},
where $\hat{\pi}_{0}:S^{*}\Lambda\to\Lambda$ is the natural projection,
and define $\hat{\omega}:=d\hat{\alpha}$ on $S^{*}\Lambda$.

Let $j=\varphi_{F}\circ\hat{i}\circ\varphi_{F^{*}}$, then for any
$X\in T_{\xi}S^{*}\Lambda$, \begin{equation}
(j^{*}\alpha)_{\xi}(X)=\alpha_{j(\xi)}(j_{*}X)=j(\xi)(\pi_{*}j_{*}X)=\xi(\hat{\pi}_{0*}X)=\hat{\alpha}_{\xi}(X)\label{eq:long}\end{equation}
in which $\alpha$ and $\omega$ on $S^{*}\mathbb{R}^{n}$ are introduced
in Section \prettyref{sub:Symplectic-Structures-on}, then \prettyref{eq:long}
implies \begin{equation}
j^{*}\alpha=\hat{\alpha},\label{eq:alpha}\end{equation}
and furthermore we have \begin{equation}
j^{*}\omega=\hat{\omega}\label{eq:jome}\end{equation}
by differentiating \prettyref{eq:alpha}.

Next, let $\hat{p}$ be the projection taking $\bar{\xi}_{x}\in S\Lambda$
to the line passing $x$ with the direction $\bar{\xi}_{x}$, and
similarly for $p$ which is described in \prettyref{eq:prodiag}.
Consider the diagram 

\begin{equation}
\begin{array}{ccccccc}
S^{*}\Lambda & \overset{\varphi_{F^{*}}}{\simeq} & S\Lambda & \overset{\hat{i}}{\hookrightarrow} & S\mathbb{R}^{n} & \overset{\varphi_{F}}{\simeq} & S^{*}\mathbb{R}^{n}\\
 &  & \downarrow\hat{p} &  & \downarrow p\\
 &  & \overline{Gr_{1}(\Lambda)} & \overset{i}{\hookrightarrow} & \overline{Gr_{1}(\mathbb{R}^{n})}\end{array}\label{eq:rect}\end{equation}
obtained by combining diagram \prettyref{eq:prodiag} and \prettyref{eq:long}.
By the definitions of the maps in \prettyref{eq:rect}, we know the
diagram is commutative. By \prettyref{thm:There-exists-a}, we have
$p^{*}\omega_{0}=\varphi_{F}^{*}\omega$ and $\hat{p}^{*}\hat{\omega}_{0}=\varphi_{F}^{*}\hat{\omega}$.
Combining with \prettyref{eq:jome} and the commutativity of the diagram
\prettyref{eq:rect}, we obtain the desired claim $i^{*}\omega_{0}=\hat{\omega}_{0}$.
\end{proof}
Suppose $N$ is a hypersurface in $(\mathbb{R}^{n},F)$, then we have
the following 
\begin{prop}
\label{pro:surfa}$vol_{n-1}(N)=\frac{1}{2\epsilon_{n-1}}\int_{l\in\overline{Gr_{1}(\mathbb{R}^{n})}}\#(N\cap l)|\omega_{0}^{n-1}|$,
where $\omega_{0}$ is the symplectic form on $\overline{Gr_{1}(\mathbb{R}^{n})}$. 
\end{prop}
This idea of intrinsic proof is given by Dr. Joseph H. G. Fu.
\begin{proof}
It suffices to prove the claim in the case when $N$ is affine. Without
loss of generality, assume $N\subset\mathbb{R}^{n-1}\subset\mathbb{R}^{n}$
is compact and convex with smooth boundary.

Consider the following diagram\begin{equation}
S^{*}N\overset{\hat{i}}{\hookrightarrow}S^{*}\mathbb{R}^{n-1}\overset{\overset{\varphi_{F^{*}}}{\cong}}{\rightarrow}S\mathbb{R}^{n-1}\overset{i}{\hookrightarrow}S\mathbb{R}^{n}\overset{\overset{\varphi_{F}}{\cong}}{\rightarrow}S^{*}\mathbb{R}^{n}\overset{\pi}{\rightarrow}\overline{Gr_{1}(\mathbb{R}^{n})},\end{equation}
where $i$ and $k$ are embeddings, and $\pi:=p\circ\varphi_{F}^{-1}=p\circ\varphi_{F^{*}}$
is a projection from diagram \prettyref{eq:prodiag}.

As $N$ is a $(n-1)$-dimensional manifold, the canonical contact
form $\hat{\alpha}$ on $S^{*}N$ is defined as $\hat{\alpha}_{\xi}(X):=\xi(\hat{\pi}_{0*}X)$
for $X\in T_{\xi}S^{*}N$, where $\hat{\pi}_{0}:S^{*}N\to N$ is the
projection. 

Let $j=\varphi_{F}\circ i\circ\varphi_{F^{*}}$, then \begin{equation}
(j^{*}\alpha)_{\hat{i}(\xi)}(\hat{i}_{*}X)=((\varphi_{F}\circ i\circ\varphi_{F^{*}})^{*}\alpha)_{\hat{i}(\xi)}(\hat{i}_{*}X)=\xi(\pi_{0*}X)=\hat{\alpha}_{\xi}(X).\end{equation}
for any $X\in T_{\xi}S^{*}N$, which implies $(\hat{i}\circ j)^{*}\alpha=\hat{i}^{*}j^{*}\alpha=\hat{\alpha}$,
and then $(i\circ j)^{*}\omega=\hat{\omega}$ where $\hat{\omega}:=d\hat{\alpha}$
and $\omega$ is introduced in Section \prettyref{sub:Symplectic-Structures-on}.

Applying Stokes' theorem, we have\begin{equation}
\begin{array}{lllll}
\int_{D^{*}N}\hat{\omega}^{n-1} & = & \int_{\partial(D^{*}N)}\hat{\alpha}\wedge\hat{\omega}^{n-2} & = & \int_{S^{*}N}\hat{\alpha}\wedge\hat{\omega}^{n-2}+\int_{\hat{\pi}_{0}^{-1}(\partial N)}\hat{\alpha}\wedge\hat{\omega}^{n-2}\\
 &  &  & = & \int_{S^{*}N}\hat{\alpha}\wedge\hat{\omega}^{n-2}\end{array}\label{eq:longeq}\end{equation}
since the degree of $\hat{\alpha}\wedge\hat{\omega}^{n-2}$ on the
compoment mesuring perturbations of base points is bigger than the
dimension of the base manifold, and\begin{equation}
\begin{array}{lll}
\int_{S_{+}^{*}\mathbb{R}^{n}\cap\pi_{0}^{-1}(N)}\omega^{n-1} & = & \int_{\partial(S_{+}^{*}\mathbb{R}^{n}\cap\pi_{0}^{-1}(N))}\alpha\wedge\omega^{n-2}\\
 & = & \int_{S^{*}N}\hat{i}^{*}j^{*}\alpha\wedge\hat{i}^{*}j^{*}\omega^{n-2}+\int_{\hat{\pi}_{0}^{-1}(\partial N)}\hat{i}^{*}j^{*}\alpha\wedge\hat{i}^{*}j^{*}\omega^{n-2}\\
 & = & \int_{S^{*}N}\hat{\alpha}\wedge\hat{\omega}^{n-2},\end{array}\end{equation}
where \begin{equation}
S_{+}^{*}\mathbb{R}^{n}=\left\{ \xi\in S^{*}\mathbb{R}^{n}:\xi(v_{0})\geqslant0,\, v_{0}\mbox{\, satisfies}\, dF(v_{0})(v)=0\,\,\mbox{for}\:\mbox{all}\, v\in S\mathbb{R}^{n-1}\right\} .\end{equation}
Therefore,\begin{equation}
\int_{D^{*}N}\hat{\omega}^{n-1}=\int_{S_{+}^{*}\mathbb{R}^{n}\cap\pi_{2}^{-1}(N)}\omega^{n-1}.\label{eq:sphlines}\end{equation}

Now let us consider the {}``upper'' half space of geodesics in $(\mathbb{R}^{n},F)$,
\begin{equation}
\overline{Gr_{1}^{+}(\mathbb{R}^{n})}:=\left\{ l(x,\eta):dF(\eta)(\eta_{0})\geqslant0,\,\eta_{0}\:\mbox{satisfies}\, dF(\eta_{0})(v)=0\,\,\mbox{for}\:\mbox{all}\, v\in S\mathbb{R}^{n-1}\right\} .\end{equation}
Since $\pi^{*}\omega_{0}=\omega$, then we get\begin{equation}
\begin{array}{lll}
\int_{S_{+}^{*}\mathbb{R}^{n}\cap\pi_{0}^{-1}(N)}\omega^{n-1} & = & \int_{S_{+}^{*}\mathbb{R}^{n}\cap\pi_{0}^{-1}(N)}\pi^{*}\omega_{0}^{n-1}\\
 & = & \int_{\pi^{-1}(l)\in S_{+}^{*}\mathbb{R}^{n}\cap\pi_{0}^{-1}(N)}\#(N\cap l)\omega_{0}^{n-1}\\
 & = & \int_{l\in\overline{Gr_{1}^{+}(\mathbb{R}^{n})}}\#(N\cap l)\omega_{0}^{n-1}.\end{array}\end{equation}
Combining with \prettyref{eq:sphlines}, we obtain\begin{equation}
\begin{array}{lll}
vol_{n-1}(N) & = & \frac{1}{\epsilon_{n-1}}\int_{D^{*}N}\hat{\omega}^{n-1}\\
 & = & \frac{1}{\epsilon_{n-1}}\int_{l\in\overline{Gr_{1}^{+}(\mathbb{R}^{n})}}\#(N\cap l)\omega_{0}^{n-1}\\
 & = & \frac{1}{2\epsilon_{n-1}}\int_{l\in\overline{Gr_{1}(\mathbb{R}^{n})}}\#(N\cap l)|\omega_{0}^{n-1}|,\end{array}\end{equation}
that finishes the proof. 
\end{proof}

\section{$k$-th Holmes-Thompson Volume and Crofton Measures\label{sec:k-th-Holmes-Thompson-Volume}}

Let us introduce a general fact first. Busemann constructed all projective
metrics $F$ for projective Finsler space $(\mathbb{R}^{n},F)$, and
it was also proved in \cite{S2} by Schneider using spherical harmonics.
\begin{thm}
(Busemann) Suppose $F$ is a projective metric on $\mathbb{R}^{n}$,
then $F(x,v)=\int_{\xi\in S^{n-1}}|\langle\xi,v\rangle|f(\xi,\langle\xi,x\rangle)\Omega_{0}$
for any $(x,v)\in T\mathbb{R}^{n}$, where $\Omega_{0}$ is the Euclidean
volume form on $S^{n-1}$ and $f$ is some continuous function on
$S^{n-1}\times\mathbb{R}$. 
\end{thm}
In fact, for the case that $(\mathbb{R}^{n},F)$ is Minkowski, we
can use a theorem on surjectivity of cosine transform, \begin{equation}
\mathcal{C}(f)(\cdot)=\int_{\xi\in S^{n-1}}|\langle\xi,\cdot\rangle|f(\xi)\Omega_{0},\label{eq:cosine}\end{equation}
of even functions from Chapter 3 of \cite{G}, 
\begin{thm}
For any even $C^{2[(n+3)/2]}$ function $g$ on $S^{n-1}$, $n\geqslant2$,
where $[\cdot]$ is the greatest integer function, there is an even
function $f$ on $S^{n-1}$ such that $\mathcal{C}(f)=g$.
\end{thm}
From it we directly obtain that there exists an even function $f$
on $S^{n-1}$, such that

\begin{equation}
L(\overline{xy})=\frac{1}{4}\int_{\xi\in S^{n-1}}|\langle\xi,\overrightarrow{xy}\rangle|f(\xi)\Omega_{0}.\label{eq:lengthfor}\end{equation}

On the other hand, for any $v=\overrightarrow{xy}$, $x,y\in(\mathbb{R}^{2},F)$,
by Proposition \prettyref{pro:The-Crofton-measure} we know\begin{equation}
F(x,v)=\frac{1}{\omega_{n-1}}\int_{l\in\overline{Gr_{1}(\mathbb{R}^{2})}}\#(\overline{xy}\cap l)|\omega_{0}|.\end{equation}

In fact, there is a relation between $\Omega_{0}$ and $\omega_{0}$.
Considering the following double fibration\begin{equation}
\overline{Gr_{1}(\mathbb{R}^{n})}\overset{\pi_{1}}{\leftarrow}\mathcal{I}\overset{\pi_{2}}{\rightarrow}\overline{Gr_{n-1}(\mathbb{R}^{n})},\label{eq:double fibration}\end{equation}
where $\mathcal{I}=\left\{ (l,H)\in\overline{Gr_{1}(\mathbb{R}^{n})}\times\overline{Gr_{n-1}(\mathbb{R}^{n})}:l\subset H\right\} $,
we have 
\begin{prop}
\label{prop:tran}$GT(f\Omega_{0}\wedge dr)=\omega_{0}$, where $GT$
is the Gelfand transform for the double fibraton \eqref{eq:double fibration}
and $r$ is the Euclidean distance of a hyperplane $H$ to the origin. \end{prop}
\begin{proof}
For any $x,y$ in any $2$-plane $\Pi\subset\mathbb{R}^{n}$, we know
that the length of $\overline{xy}$,\begin{equation}
L(\overline{xy})=\frac{1}{4}\int_{l\in\overline{Gr_{1}(\Pi)}}\#(\overline{xy}\cap l)|\omega_{0}|_{\overline{Gr_{1}(\Pi)}}.\label{eq:lengthsym}\end{equation}
where $\overline{Gr_{1}(\Pi)}:=\left\{ l\in\overline{Gr_{1}(\mathbb{R}^{n})}:l\subset\Pi\right\} $.

Let $I=\left\{ l\in\overline{Gr_{1}(\Pi)}\subset\overline{Gr_{1}(\mathbb{R}^{n})}:\overline{xy}\cap l\neq\phi\right\} $
and $G_{H}:=\pi_{1}(\pi_{2}^{-1}(H))$ for $H\in\overline{Gr_{n-1}(\mathbb{R}^{n})}$.
By the fundamental theorem of Gelfand transform, \prettyref{thm:fundamental},\begin{equation}
\int_{H\in\overline{Gr_{n-1}(\mathbb{R}^{n})}}\#(I\cap G_{H})|f\Omega_{0}\wedge dr|=\int_{I}|GT(f\Omega_{0}\wedge dr)|.\end{equation}
 Therefore,\begin{equation}
\begin{array}{lll}
\int_{l\in\overline{Gr_{1}(\Pi)}}\#(\overline{xy}\cap l)|GT(f\Omega_{0}\wedge dr)| & = & \int_{I}|GT(f\Omega_{0}\wedge dr)|\\
 & = & \int_{H\in\overline{Gr_{n-1}(\mathbb{R}^{n})}}\#(I\cap G_{H})|f\Omega_{0}\wedge dr|\\
 & = & \int_{\xi\in S^{n-1}}|\langle\xi,\overrightarrow{xy}\rangle|f(\xi)\Omega_{0}\end{array}\end{equation}
since $\overline{Gr_{n-1}(\mathbb{R}^{n})}\cong S^{n-1}\times\mathbb{R}$.
By \prettyref{eq:lengthfor}and \prettyref{eq:lengthsym} we thus
obtain\begin{equation}
\int_{l\in\overline{Gr_{1}(\Pi)}}\#(\overline{xy}\cap l)|GT(f\Omega_{0}\wedge dr)|=\int_{l\in\overline{Gr_{1}(\Pi)}}\#(\overline{xy}\cap l)|\omega_{0}|,\end{equation}
 which implies $GT(f\Omega_{0}\wedge dr)|_{\overline{Gr_{1}(\Pi)}}=\omega_{0}|_{\overline{Gr_{1}(\Pi)}}$
for any plane $\Pi\subset\mathbb{R}^{n}$ by the injectivity of cosine
transform \prettyref{eq:cosine}.(In Chapter 3 of \cite{G} Groemer
shows by using condensed harmonic expansion and Parseval's equation,
that $\mathcal{C}(f_{1})=\mathcal{C}(f_{2})$ iff $f_{1}^{+}=f_{2}^{+}$,
where $f_{1}^{+}(v)=\frac{f_{1}(v)+f_{1}(-v)}{2}$ and similarly for
$f_{2}^{+}$, for any bounded integrable functions $f_{1}$ and $f_{2}$
on $S^{n-1}$.)

Now define a basis for $T_{l}\overline{Gr_{1}(\mathbb{R}^{n})}$,
the tangent space of $\overline{Gr_{1}(\mathbb{R}^{n})}$ at $l\in\overline{Gr_{1}(\mathbb{R}^{n})}$. 

Note that $\overline{Gr_{1}(\mathbb{R}^{n})}\overset{\psi}{\simeq}TS_{F}^{n-1}$
from \prettyref{sec:The-Symplectic-Structure}. Let $\left\{ e_{i}:i=1,\cdots,n\right\} $
be the basis for $\mathbb{R}^{n}$, and curve $\overline{\gamma}_{i}$
with $\overline{\gamma}_{i}(t)=l+te_{i}$ for $i=1,\cdots,n-1$, where
$l\in\overline{Gr_{1}(\mathbb{R}^{n})}$, and then define $\overline{e}_{i}:=\overline{\gamma}'_{i}(0)$
for $i=1,\cdots,n-1$. Let $l(x,\xi)$ be a line in $\overline{Gr_{1}(\mathbb{R}^{n})}$
passing through $x$ with direction $\xi$ and $r_{i}(t)(\xi)$ for
be the rotation about origin with the direction from $e_{n}$ towards
$e_{i}$ for time $t$, then let $v_{i}(t)$ be the parallel transport
from $\psi(l(x,\xi)$ along on $r_{i}(t)(\xi)$ on $S_{F}^{n-1}$,
and then define curves $\overline{\overline{\gamma}}_{i}(t)=\psi^{-1}(v_{i}(t))$
for $i=1,\cdots,n-1$, thus we can define $\overline{\overline{e}}_{i}:=\overline{\overline{\gamma}}'_{i}(0)$.
Then $\left\{ \overline{e}_{i},\overline{\overline{e}}_{j}:i,j=1,\cdots,n-1\right\} $
is a basis for $T_{l}\overline{Gr_{1}(\mathbb{R}^{n})}$. 

Here we have four cases to discuss.

First of all, one can obtain the fact \begin{equation}
GT(f\Omega_{0}\wedge dr)(\overline{e}_{i},\overline{\overline{e}}_{i})=\omega_{0}(\overline{e}_{i},\overline{\overline{e}}_{i})\end{equation}
 by choosing a plane $\Pi_{i}$ with the tangent space of $\overline{Gr_{1}(\Pi_{i})}$
spanned by $\overline{e}_{i}$ and $\overline{\overline{e}}_{i}$
for $i=1,\cdots,n-1$. 

On the other hand, in the double fibration \prettyref{eq:double fibration},
$\pi_{2}|_{\pi_{1}^{-1}(L_{ij})}$, in which $L_{ij}$ is be the lines
in $\overline{Gr_{1}(\mathbb{R}^{n})}$ obtained by translation along
$\overline{e}_{i}$ or $\overline{e}_{j}$ for $i,j=1,\cdots,n-1$,
is not a submersion from $\pi_{1}^{-1}(L_{ij})$ to $\overline{Gr_{n-1}(\mathbb{R}^{n})}$.
Precisely, choose $\tilde{e}_{i}$ and $\tilde{e}_{j}$ in $T_{(l,H)}\mathcal{I}$,
$l\subset H$ such that $d\pi_{1}(\tilde{e}_{i})=\overline{e}_{i}$
and $d\pi_{1}(\tilde{e}_{j})=\overline{e}_{j}$, moreover, $d\pi_{2}(\tilde{e}_{i})$
and $d\pi_{2}(\tilde{e}_{i})$ are linearly dependent in $T_{H}\overline{Gr_{n-1}(\mathbb{R}^{n})}$.
Therefore $\pi_{1*}\pi_{2}^{*}(f\Omega_{0}\wedge dr)_{l}(\overline{e}_{i},\overline{e}_{j})=\int_{\pi_{1}^{-1}(l)}\pi_{2}^{*}(f\Omega_{0}\wedge dr)(\overline{e}_{i},\overline{e}_{j})=0$
for $i,j=1,\cdots,n-1$, and obviously $\omega_{0}(\overline{e}_{i},\overline{e}_{j})=0$,
thus \begin{equation}
GT(f\Omega_{0}\wedge dr)(\overline{e}_{i},\overline{e}_{j})=\omega_{0}(\overline{e}_{i},\overline{e}_{j})=0\end{equation}
 for $i,j=1,\cdots,n-1$. 

For the case of $\overline{e}_{i}$ and $\overline{\overline{e}}_{j}$,
$i\neq j$, $i=1,\cdots,n-1$. Let $\bar{L}_{ij}$ be the lines in
$\overline{Gr_{1}(\mathbb{R}^{n})}$ obtained by translation along
$\overline{e}_{i}$ or rotation along $\overline{\overline{e}}_{j}$.
Again, $\pi_{2}|_{\bar{L}_{ij}}$ in \prettyref{eq:double fibration}
is not a submersion from $\pi_{1}^{-1}(\bar{L}_{ij})$ to $\overline{Gr_{n-1}(\mathbb{R}^{n})}$
either, and it also can be explained precisely as the above case,
therefore $\pi_{1*}\pi_{2}^{*}(f\Omega_{0}\wedge dr)(\overline{e}_{i},\overline{\overline{e}}_{j})=0$
for $i,j=1,\cdots,n-1$, and obviously $\omega_{0}(\overline{e}_{i},\overline{\overline{e}}_{j})=0$,
thus \begin{equation}
GT(f\Omega_{0}\wedge dr)(\overline{e}_{i},\overline{\overline{e}}_{j})=\omega_{0}(\overline{e}_{i},\overline{\overline{e}}_{j})=0\end{equation}
 for $i\neq j$, $i,j=1,\cdots,n-1$. 

Similarly for the last case of $\overline{\overline{e}}_{i}$ and
$\overline{\overline{e}}_{j}$, $i,j=1,\cdots,n-1$, \begin{equation}
GT(f\Omega_{0}\wedge dr)(\overline{\overline{e}}_{i},\overline{\overline{e}}_{j})=\omega_{0}(\overline{\overline{e}}_{i},\overline{\overline{e}}_{j})=0.\end{equation}

So we have $GT(f\Omega_{0}\wedge dr)=\omega_{0}$ on $\overline{Gr_{1}(\mathbb{R}^{n})}$.
\end{proof}
One can use the diagonal intersection map and Gelfand transform by
following \cite{AF2} to construct Crofton measure for the $k$-th
Holmes-Thompson volume. 

Let $\Omega_{n-1}:=f\Omega_{0}\wedge dr$ and define a map\begin{equation}
\begin{array}{c}
\pi:\overline{Gr_{n-1}(\mathbb{R}^{n})}^{k}\backslash\triangle_{k}\rightarrow\overline{Gr_{n-k}(\mathbb{R}^{n})}\\
\pi((H_{1},\cdots,H_{k}))=H_{1}\cap\cdots\cap H_{k},\end{array}\end{equation}
where $\triangle_{k}=\left\{ (H_{1},\cdots,H_{k}):dim(H_{1}\cap\cdots\cap H_{k})>n-k\right\} $
and then let $\Omega_{n-k}:=\pi_{*}\Omega_{n-1}^{k}$. 

Now consider the following double fibration,\begin{equation}
\overline{Gr_{1}(\mathbb{R}^{n})}\overset{\pi_{1,k}}{\leftarrow}\mathcal{I}_{k}\overset{\pi_{2,k}}{\rightarrow}\overline{Gr_{n-k}(\mathbb{R}^{n})},\label{eq:kdf}\end{equation}
 where $\mathcal{I}_{k}=\left\{ (l,S)\in\overline{Gr_{1}(\mathbb{R}^{n})}\times\overline{Gr_{n-k}(\mathbb{R}^{n})}:l\subset S\right\} $.
Then we have the following proposition about the Gelfand transform
on \prettyref{eq:kdf}
\begin{prop}
\label{prop:genearcase}$GT(\Omega_{n-k})=\omega_{0}^{k}$ for $1\leq k\leq n-1$.\end{prop}
\begin{proof}
Let \begin{equation}
\mathcal{H}:=\left\{ (l,(H_{1},H_{2},\cdots,H_{k}))\in\overline{Gr_{1}(\mathbb{R}^{n})}\times\overline{Gr_{n-1}(\mathbb{R}^{n})}^{k}:l\subset H_{1}\cap\cdots\cap H_{k}\right\} \end{equation}
and consider the following diagram\begin{equation}
\begin{array}{lllll}
\overline{Gr_{1}(\mathbb{R}^{n})} & \overset{\pi_{1,k}}{\leftarrow} & \mathcal{I}_{k} & \overset{\pi_{2,k}}{\rightarrow} & \overline{Gr_{n-k}(\mathbb{R}^{n})}\\
 & \tilde{\pi}_{1}^{\nwarrow} & \uparrow\tilde{\pi} &  & \uparrow\pi\\
 &  & \mathcal{H} & \overset{\tilde{\pi}_{2}}{\rightarrow} & \overline{Gr_{n-1}(\mathbb{R}^{n})}^{k},\end{array}\label{eq:trian}\end{equation}
in which $\tilde{\pi}:\mathcal{H}\rightarrow\mathcal{I}_{k}$ is defined
by $\tilde{\pi}((l,(H_{1},H_{2},\cdots,H_{k})))=(l,H_{1}\cap H_{2}\cap\cdots\cap H_{k}))$. 

Note that \begin{equation}
\pi_{1*}\pi_{2}^{*}\Omega_{n-1}=\omega_{0},\label{eq:origel}\end{equation}
by Proposition \prettyref{prop:tran}. 

For the lower part of the diagram \prettyref{eq:trian}, \begin{equation}
\overline{Gr_{1}(\mathbb{R}^{n})}\overset{\tilde{\pi}_{1}}{\leftarrow}\mathcal{H}\overset{\tilde{\pi}_{2}}{\rightarrow}\overline{Gr_{1}(\mathbb{R}^{n})}^{k},\end{equation}
By manipulating the map $\tilde{\pi}_{2}=\underset{k}{\underbrace{\pi_{2}\times\cdots\times\pi_{2}}}$,
the product of $k$ copies of the map $\pi_{2}$, applying Fubini
theorem for \prettyref{eq:origel} and using the fact that $\tilde{\pi}_{1}\times\tilde{\pi}_{2}:\mathcal{H}\rightarrow\overline{Gr_{1}(\mathbb{R}^{n})}\times\overline{Gr_{1}(\mathbb{R}^{n})}^{k}$
is an immersion, one can infer $\tilde{\pi}_{1*}\tilde{\pi}_{2}^{*}\Omega_{n-1}^{k}=\omega_{0}^{k}$. 

Thus, by the commutativity of the diagram \prettyref{eq:trian} we
obtain $\pi_{1,k*}\pi_{2,k}^{*}\Omega_{n-k}=\omega_{0}^{k}$.
\end{proof}
In order to study the $k$-th Holmes-Thompson volume, one can restrict
on some $k+1$-dimensional flat subspace. So fix $S\in\overline{Gr_{k+1}(\mathbb{R}^{n})}$
and then define a map by intersection\begin{equation}
\begin{array}{c}
\pi_{S}:\overline{Gr_{n-k}(\mathbb{R}^{n})}\setminus\triangle(S)\rightarrow\overline{Gr_{1}(S)}\\
\pi_{S}(H^{n-k})=H^{n-k}\cap S\end{array}\label{eq:interma}\end{equation}
for $H^{n-k}\in\overline{Gr_{n-k}(\mathbb{R}^{n})}\setminus\triangle(S)$,
where \begin{equation}
\triangle(S):=\left\{ H^{n-k}\in\overline{Gr_{n-k}(\mathbb{R}^{n})}:dim(H^{n-k}\cap S)>0\right\} .\end{equation}
Then we have the following proposition 
\begin{prop}
\label{pro:interse}$(\pi_{S})_{*}\Omega_{n-k}=\omega_{0}^{k}|_{\overline{Gr_{1}(S)}}$,
for $1\leq k\leq n-1$.\end{prop}
\begin{proof}
From Proposition \prettyref{prop:genearcase}, we know that $\pi_{1,k*}\pi_{2,k}^{*}\Omega_{n-k}=\omega_{0}^{k}$
for the double fibration $\overline{Gr_{n-k}(\mathbb{R}^{n})}\overset{\pi_{2,k}}{\leftarrow}\mathcal{I}_{k}\overset{\pi_{1,k}}{\rightarrow}\overline{Gr_{1}(\mathbb{R}^{n})}$. 

Therefore, one can obtain by the definition of the intersection map
\prettyref{eq:interma} \begin{equation}
(\pi_{S})_{*}\Omega_{n-k}=\pi_{1,k*}\pi_{2,k}^{*}\Omega_{n-k}|_{\overline{Gr_{1}(S)}}=\omega_{0}^{k}|_{\overline{Gr_{1}(S)}}.\end{equation}

\end{proof}
Finally, one can obtain the following theorem about Holmes-Thompson
volumes.
\begin{thm}
\label{thm:resu}(Alvarez) Suppose $N$ is a $k$-dimensional submanifold
in $(\mathbb{R}^{n},F)$. Then $vol_{k}(N)=\frac{1}{2\epsilon_{k}}\int_{P\in\overline{Gr_{n-k}(\mathbb{R}^{n})}}\#(N\cap P)|\Omega_{n-k}|$
for $1\leq k\leq n-1$.\end{thm}
\begin{proof}
By Proposition \prettyref{pro:surfa}, the claim is true for hypersurface
case.

It is sufficient to show the claim for the case when $N\subset S$
for some $S\in\overline{Gr_{k+1}(\mathbb{R}^{n})}$. We obtain by
Proposition \prettyref{pro:surfa} and Proposition \prettyref{pro:interse},\begin{equation}
\begin{array}{lll}
vol_{k}(N) & = & \frac{1}{2\epsilon_{k}}\int_{l\in\overline{Gr_{1}(S)}}\#(N\cap l)|\omega_{0}^{k}|\\
 & = & \frac{1}{2\epsilon_{k}}\int_{l\in\overline{Gr_{1}(S)}}\#(N\cap l)|(\pi_{S})_{*}\Omega_{n-k}|\\
 & = & \frac{1}{2\epsilon_{k}}\int_{P\in\overline{Gr_{n-k}(\mathbb{R}^{n})}}\#(N\cap P)|\Omega_{n-k}|.\end{array}\end{equation}
as desired.
\end{proof}

\section{Length and Related}

The classic Crofton formula is \begin{equation}
Length(\gamma)=\frac{1}{4}\int_{0}^{\infty}\int_{0}^{2\pi}\#(\gamma\cap l(r,\theta))d\theta dr\end{equation}
for any rectifiable curve in Euclidean plane, where $\theta$ is the
angle of the normal of the oriented line $l$ to the $x$-axis and
$r$ is its distance to the origin. Let us denote the affine $1$-Grassmannians
(lines) in $\mathbb{R}^{2}$ by $\overline{Gr_{1}(\mathbb{R}^{2})}$
.

As for Minkowski plane, it is a normed two dimensional space with
a norm $F(\cdot)=||\cdot||$ , in which the unit disk is convex and
$F$ has some smoothness.

Two of the key tools used to obtain the Crofton formula for Minkowski
plane are the cosine transform and Gelfand transform. Let us explain
them one by one first and see their connection next. A fact from spherical
harmonics about cosine transform is there is some even function on
$S^{1}$ such that \begin{equation}
F(\cdot)=\frac{1}{4}\int_{S^{1}}|\langle\xi,\cdot\rangle|g(\xi)d\xi,\label{eq:exist}\end{equation}
if $F$ is an even $C^{4}$ function on $S^{1}$. A good reference
for this is \cite{G}. As for Gelfand transform, it is the transform
of differential forms and densities on double fibrations, for instance,
$\mathbf{\mathbb{R}}^{2}\stackrel{\pi_{1}}{\leftarrow}\mathcal{I}\stackrel{\pi_{2}}{\rightarrow}\overline{Gr_{1}(\mathbf{\mathbb{R}^{2})}}$,
where $\mathcal{I}:=\left\{ (x,l)\in\mathbf{\mathbb{R}}^{2}\times\overline{Gr_{1}(\mathbf{\mathbb{R}^{2})}}:x\in l\right\} $
is the incidence relations and $\pi_{1}$ and $\pi_{2}$ are projections.
A formula one can take as an example of the fundamental theorem of
Gelfand transform is the following

\begin{equation}
\int_{\gamma}\pi_{1*}\pi_{2}^{*}|\Omega|=\int_{l\in\overline{Gr_{1}(\mathbb{R}^{2})}}\#(\gamma\cap l)|\Omega|,\label{eq:gl}\end{equation}
where $\Omega:=g(\theta)d\theta\wedge dr$. But we give a direct proof
here.
\begin{proof}
First, consider the case of $\Omega=d\theta\wedge dr$. For any $v\in T_{x}\gamma$,
since there is some $v'\in T_{x'}\mathcal{I}$, such that $\pi_{1*}(v')=v$,
then

\begin{equation}
\begin{array}{lll}
(\pi_{1*}\pi_{2}^{*}|\Omega|)_{x}(v) & = & (\int_{\pi_{1}^{-1}(x)}\pi_{2}^{*}|\Omega|)_{x}(v)\\
 & = & \int_{x'\in\pi_{1}^{-1}(x)}(\pi_{2}^{*}|\Omega|)_{x'}(v')\\
 & = & \int_{S^{1}}(\pi_{2}^{*}|d\theta\wedge dr|)(v')\\
 & = & \int_{S^{1}}|dr(\pi_{2*}(v'))|d\theta\\
 & = & \int_{S^{1}}|\langle v,\theta\rangle|d\theta\\
 & = & 4|v|.\end{array}\end{equation}
So $\int_{\gamma}\pi_{1*}\pi_{2}^{*}|\Omega|=4Length(\gamma)=\int_{l\in\overline{Gr_{1}(\mathbb{R}^{2})}}\#(\gamma\cap l)|\Omega|$
by the classic Crofton formula.

When $\Omega=f(\theta)d\theta\wedge dr$, we just need to replace
$d\theta$ by $g(\theta)d\theta$ in the equalities in the first case. 
\end{proof}
Moreover, from the above proof and \prettyref{eq:exist}, for any
curve $\gamma(t):[a,b]\rightarrow\mathbb{R}^{2}$ differentiable almost
everywhere in the Minkowski space, \begin{equation}
\int_{\gamma}\pi_{1*}\pi_{2}^{*}|\Omega|=\int_{a}^{b}(\pi_{1*}\pi_{2}^{*}|\Omega|)(\gamma'(t))dt=\int_{a}^{b}4F(\gamma'(t))dt=4Length(\gamma),\end{equation}
so then by \prettyref{eq:gl} we know \begin{equation}
Length(\gamma)=\frac{1}{4}\int_{l\in\overline{Gr_{1}(\mathbb{R}^{2})}}\#(\gamma\cap l)|g(\theta)d\theta\wedge dr|\label{eq:leng}\end{equation}
for Minkowski plane.

The Holmes-Thompson Area $HT^{2}(U)$ of a measurable set $U$ in
a Minkowski plane is defined as $HT^{2}(U):=\frac{1}{\pi}\int_{D^{*}U}|\omega_{0}|^{2}$,
where $\omega_{0}$ is the natural symplectic form on the cotangent
bundle of $\mathbb{R}^{2}$ and $D^{*}U:=\left\{ (x,\xi)\in T^{*}\mathbb{R}^{2}:F^{*}(\xi)\leq1\right\} $.
To study it from the perspective of integral geometry, we need to
introduce a symplectic form $\omega$ on the space of affine lines
$\overline{Gr_{1}(\mathbf{\mathbb{R}^{2})}}$, that one can see \cite{Helbert}.

\section{HT Area and Related}

Now let's see the Crofton formula for Minkowski plane, which is $Length(\gamma)=\frac{1}{4}\int_{\overline{Gr_{1}(\mathbb{R}^{2})}}\#(\gamma\cap l)|\omega|$.
To prove this, it is sufficient to show that it holds for for any
straight line segment \begin{equation}
L:[0,||p_{2}-p_{2}||]\rightarrow\mathbb{R}^{2},\, L(t)=p_{1}+\frac{p_{2}-p_{1}}{||p_{2}-p_{1}||}t,\end{equation}
starting at $p_{1}$ and ending at $p_{2}$ in $\mathbb{R}^{2}$.
First, using the diffeomorphism between the circle bundle and co-circle
bundle, which is \begin{equation}
\begin{array}{c}
\varphi_{F}:S\mathbb{R}^{2}\rightarrow S^{*}\mathbb{R}^{2}\\
\varphi_{F}(x,\xi)=(x,dF_{\xi}),\end{array}\end{equation}
we can obtain a fact that\begin{equation}
\begin{array}{lll}
\int_{L\times\left\{ \frac{p_{2}-p_{1}}{||p_{2}-p_{1}||}\right\} }\varphi_{F}^{*}\alpha_{0} & = & \int_{\varphi_{F}(L\times\left\{ \frac{p_{2}-p_{1}}{||p_{2}-p_{1}||}\right\} )}\alpha_{0}\\
 & = & \int_{0}^{||p_{2}-p_{1}||}\alpha_{0dF_{\frac{p_{2}-p_{1}}{||p_{2}-p_{1}||}}}((\frac{p_{2}-p_{1}}{||p_{2}-p_{1}||},0))dt\\
 & = & \int_{0}^{||p_{2}-p_{1}||}dF_{\frac{p_{2}-p_{1}}{||p_{2}-p_{1}||}}(\frac{p_{2}-p_{1}}{||p_{2}-p_{1}||})dt,\end{array}\end{equation}
where $\alpha_{0}$ is the tautological one-form, precisely $\alpha_{0\xi}(X):=\xi(\pi_{0*}X)$
for any $X\in T_{\xi}T^{*}\mathbb{R}^{2}$, and $d\alpha_{0}=\omega_{0}$.
Applying the the basic equality that $dF_{\xi}(\xi)=1$, which is
derived from the positive homogeneity of $F$, for all $\xi\in S\mathbb{R}^{2}$,
the above quantity becomes $\int_{0}^{||p_{2}-p_{1}||}1dt$, which
equals to $||p_{2}-p_{1}||$.

Let $R:=\left\{ \xi_{x}\in S^{*}\mathbb{R}^{2}:x\in\overline{p_{1}p_{2}}\right\} $
and $T=\left\{ l\in\overline{Gr_{1}(\mathbf{\mathbb{R}^{2})}}:l\cap\overline{p_{1}p_{2}}\neq\textrm{Ø}\right\} $,
and $p'$ is the projection (composition) from $S^{*}\mathbb{R}^{2}$
to $\overline{Gr_{1}(\mathbf{\mathbb{R}^{2})}}$.

Apply the above fact and $p'^{*}\omega=\omega_{0}$,\begin{equation}
\begin{array}{lllllll}
\int_{T}|\omega| & = & \int_{p'(R)}|\omega| & = & \int_{R}|p'^{*}\omega| & = & \int_{R}|\omega_{0}|\\
 &  &  &  &  & = & |\int_{R^{+}}\omega_{0}|+|\int_{R^{-}}\omega_{0}|\\
 &  &  &  &  & = & |\int_{\partial R^{+}}\alpha_{0}|+|\int_{\partial R^{-}}\alpha_{0}|\\
 &  &  &  &  & = & 4||p_{2}-p_{1}||.\end{array}\end{equation}
Thus we have shown the Crofton formula for Minkowski plane.

Furthermore, combining with \prettyref{eq:leng}, we have \begin{equation}
\frac{1}{4}\int_{l\in\overline{Gr_{1}(\mathbb{R}^{2})}}\#(\gamma\cap l)|\Omega|=\frac{1}{4}\int_{\overline{Gr_{1}(\mathbb{R}^{2})}}\#(\gamma\cap l)|\omega|,\end{equation}
where $\Omega=g(\theta)d\theta\wedge dr$. Then, by the injectivity
of cosine transform in \cite{G}, $|\Omega|=|\omega|$.

To obtain the HT area, one can define a map

\begin{equation}
\begin{array}{c}
\pi:\overline{Gr_{1}(\mathbb{R}^{2})}\times\overline{Gr_{1}(\mathbb{R}^{2})}\setminus\tilde{\triangle}\rightarrow\mathbb{R}^{2}\\
\pi(l,l')=l\cap l',\end{array}\end{equation}
where $\tilde{\triangle}:=\left\{ (l,l'):l\parallel l'\mbox{\, or\,}l=l'\right\} $,
extended from Alvarez's construction of taking intersections, \cite{AF}.
The following theorem can be obtained.
\begin{thm}
$HT^{2}(U)=\frac{1}{2\pi}\int_{x\in\mathbb{R}^{2}}\chi(x\cap U)|\pi_{*}\Omega^{2}|$
for any bounded measurable subset $U$ of a Minkowski plane.\end{thm}
\begin{proof}
On one hand, \begin{equation}
\frac{1}{\pi}\int_{D^{*}U}\omega_{0}^{2}=\frac{1}{\pi}\int_{\partial D^{*}U}\omega_{0}^{2}=\frac{1}{\pi}\int_{S^{*}U}\alpha_{0}\wedge\omega_{0}.\label{eq:e1}\end{equation}
On the other hand, let $\mathcal{T}_{U}:=\left\{ ((l,l')\in\overline{Gr_{1}(\mathbb{R}^{2})}\times\overline{Gr_{1}(\mathbb{R}^{2})}:l\cap l'\in U\right\} $$,$\begin{equation}
\frac{1}{\pi}\int_{x\in\mathbb{R}^{2}}\chi(x\cap U)\pi_{*}\Omega^{2}=\frac{1}{\pi}\int_{U}\pi_{*}\omega^{2}=\frac{1}{\pi}\int_{\mathcal{T}_{U}}\omega^{2}.\label{eq:e2}\end{equation}
Let $\mathbb{T}^{*}U:=\left\{ (\xi_{x},\xi_{x}'):\xi_{x},\xi_{x}'\in S_{x}^{*}U\right\} $,
then \begin{equation}
(p'\times p')^{-1}(\mathcal{T}_{U})=\mathbb{T}^{*}U\setminus\left\{ (\xi_{x},\xi_{x}):\xi_{x}\in S_{x}^{*}U\right\} .\end{equation}
Therefore\begin{equation}
\begin{array}{lll}
\frac{1}{\pi}\int_{\mathcal{T}_{U}}\omega^{2} & = & \frac{1}{\pi}\int_{\mathbb{T}^{*}U\setminus\left\{ (\xi_{x},\xi_{x}):\xi_{x}\in S_{x}^{*}U\right\} }p'^{*}\omega^{2}\\
 & = & \frac{1}{\pi}\int_{\mathbb{T}^{*}U\setminus\left\{ (\xi_{x},\xi_{x}):\xi_{x}\in S_{x}^{*}U\right\} }\omega_{0}^{2}\\
 & = & \frac{2}{\pi}\int_{\left\{ (\xi_{x},\xi_{x}):\xi_{x}\in S_{x}^{*}U\right\} }\alpha_{0}\wedge\omega_{0}\\
 & = & \frac{2}{\pi}\int_{S^{*}U}\alpha_{0}\wedge\omega_{0}.\end{array}\label{eq:e3}\end{equation}
So the claim follows from \prettyref{eq:e1},\prettyref{eq:e2} and
\prettyref{eq:e3}.\end{proof}
\begin{acknowledgement*}
Thanks to J. Fu for some helpful discussions on this subject.\end{acknowledgement*}

\address{Department of Mathematics, University of Georgia, Athens, GA 30602,
U.S.A.}

\email{yliu@math.uga.edu}
\end{document}